\documentclass[12pt]{amsart}
 \usepackage[margin=1.4in]{geometry}
 \usepackage{amsmath,amssymb,amsthm,graphicx,amsxtra, setspace}
 \usepackage[utf8]{inputenc}
 \usepackage{mathrsfs}
 \usepackage{hyperref}
 \usepackage{xcolor}
 \usepackage{upgreek}
 \usepackage{float}
 \usepackage{epsfig}
 \usepackage{mathtools}
 \allowdisplaybreaks
 
 \newtheorem{theorem}{Theorem}[section]

 \newtheorem{definition}{Definition}[section]
 
 \theoremstyle{definition}
\newtheorem{example}{Example}[section]
 
 \usepackage{latexsym}
 \usepackage{hyperref}
 \usepackage{graphicx}
 \usepackage{epstopdf}

\usepackage{amssymb,bbm,enumerate,bbm,amsmath}
\usepackage{changepage}
\usepackage[a4paper, left=1.3in, right=.3in]{}

\usepackage{mathrsfs}
\usepackage{tikz}
\usepackage{ dsfont}
\usepackage{hyperref}
\usetikzlibrary{arrows}
\usepackage{amsthm}
\usepackage{array,amsfonts}
\usepackage{amsmath}
\usepackage[utf8]{inputenc}
\usepackage[T1]{fontenc}
\usepackage{mathtools}
\usepackage{enumitem}
\usepackage[thinc]{esdiff}
\usepackage{multirow}
\usepackage{graphics}
\usepackage{amsmath,bm}

\newtheorem{lemma}[theorem]{Lemma}
\newtheorem{corollary}[theorem]{Corollary}

\newtheorem{note}{Note}
\newtheorem{case}{Case}

 \allowdisplaybreaks
 
 \let\originalleft\left
 \let\originalright\right
 \renewcommand{\left}{\mathopen{}\mathclose\bgroup\originalleft}
 \renewcommand{\right}{\aftergroup\egroup\originalright}

\makeatletter
\renewcommand{\subsection}{\@startsection{subsection}{2}%
  {\z@}{-3.25ex \@plus -1ex \@minus -.2ex}%
  {1ex \@plus .2ex}{\normalfont\bfseries\newline}}
\makeatother

 \newcommand{\Addresses}{{
 		\footnote{

 			\noindent	 \textsuperscript{1,2} Department of Mathematics, Indian Institute of Technology Roorkee, Roorkee, 247667, India.	
 			
 			\noindent  \textit{e-mail\textsuperscript{1}:} \texttt{p\_yadav@ma.iitr.ac.in}
 			
 			\noindent  \textit{e-mail\textsuperscript{2}:} \texttt{tanuja.srivastava@ma.iitr.ac.in}

 }}}
 
\begin{document}
 	\title[]{Likelihood geometry of the Gumbel's Type-I bivariate exponential distribution\Addresses}
 	\author [Pooja Yadav.  Tanuja Srivastava]{Pooja Yadav\textsuperscript{1}.  Tanuja Srivastava\textsuperscript{2}}
 	\maketitle
	\begin{abstract}
    	In algebraic statistics, the maximum likelihood degree of a statistical model refers to the number of solutions (counted with multiplicity) of the score equations over the complex field. In this paper, the maximum likelihood degree of the association parameter of Gumbel's Type-I bivariate exponential distribution is investigated using algebraic techniques.
	\end{abstract}
  \maketitle	
	\section{Introduction}

	The exponential distribution is the most helpful tool for lifetime models, common in many real-world applications such as queueing theory, risk analysis, medical science, and environmental studies \cite{BBa}. Since the exponential distribution has many real-world applications, there is a growing interest in the bivariate extension of this distribution. In the literature, several bivariate exponential distributions have been proposed by many authors \cite{KBJ}, \cite{LB}, \cite{MO}. Bivariate exponential distributions are primarily used in many applications, such as survival analysis and telecommunications \cite{Ha}, \cite{LB}, \cite{NK}. In $1960$, Gumbel introduced three bivariate exponential distributions called the Type-I (GBED-I), Type-II (GBED-II) and Type-III (GBED-III), in which the marginal distribution of each variable follows exponential distribution \cite{Gu}. Gumbel's bivariate exponential distribution is used in many applications in reliability engineering \cite{DNR}, \cite{PM}. 
	
	This paper will focus on Gumbel's Type-I bivariate exponential distribution (GBED-I). The maximum likelihood estimation of the parameter of GBED-I was initially attempted by Barnett in \cite{Ba} but could not provide a conclusive solution. Further, the maximum likelihood estimator of the parameter of GBED-I using ranked set sampling, generalized modified ranked set sampling and extreme ranked set sampling is attempted by Sevil and Yildiz in \cite{SY}.  In both works, the closed-form solution could not be obtained.
	
	The maximum likelihood degree (ML-degree) of a statistical model is the number of critical points of the log-likelihood function over the complex field. Knowledge of the ML-degree is important when applying the numerical algebraic-geometric method to get solutions to the score equations in the maximum likelihood estimation problem. The maximum likelihood estimator is one of the solutions of the score equations. When the score equations are rational, algebraic techniques can be used to find the maximum likelihood estimators. In algebraic statistics, the ML-degree for multinomial probabilities and parameters of multivariate Gaussian distribution, especially different structures of variance-covariance matrix, are studied in \cite{ABBG}, \cite{CMR}, \cite{LNRW} and \cite{MMW}. The score equations of discrete and Gaussian exponential family distributions are rational, and the ML-degree for these exponential family distributions is constant for generic data. If the ML-degree is one, it can be expressed equivalently by saying that the maximum likelihood estimator exists and is a rational function of data \cite{DSS}.  The readers are referred to \cite{CHKS} and \cite{S} for more details about the ML-degree for generic data.

	In this paper, we calculate the maximum likelihood degree (ML-degree) of the parameter of Gumbel's Type-I bivariate exponential distribution (GBED-I) based on the given data. Although GBED-I is not a part of the exponential family of distributions, its score equation is rational, which allows us to use algebraic techniques to compute the ML-degree of the parameter of GBED-I. An important observation is that the ML-degree of the parameter of GBED-I depends on the number of data points and the data itself. We find that the ML-degree of the parameter of GBED-I is always greater than one and can be at most twice the sample size, depending on the data's structures.
	
	The paper is organized as follows, in section 2, the maximum likelihood estimation problem of the parameter of GBED-I is introduced and shown that the score equation is a rational function of the parameter and the ML-degree of the parameter of GBED-I is defined. In section 3, the geometry of the score equation is explained. In section 4, the multiplicity of common zeros of numerator and denominator functions of the score equation is counted. In section 5, the ML-degree of the parameter of GBED-I is calculated. Finally, the paper is concluded in the last section.

	\section{Preliminaries}
	The probability density function (PDF) of Gumbel's Type-I bivariate exponential distributed random vector $X=(x,y)^\top$, $(x, y)$ in the first quadrant of $\mathbb{R}^2$, with association parameter $\theta$, is
	\begin{equation*}
		f(x, y)=  e^{-(x+y+\theta xy)} [(1+\theta x)(1+\theta y)-\theta],
	\end{equation*}
	where, $ \theta \in [0,1] \subset \mathbb{R}$, and the marginal distributions of each $x$ and $y$ are standard exponential. If $\theta =0 $, then $x$ and $y$ are mutually independent \cite{KBJ}. For more properties and results of GBED-I, see \cite{BBe}, \cite{CSH}, \cite{KBJ}.	GBED($\theta$) denotes Gumbel's Type-I bivariate exponential distribution with $\theta$ as an association parameter. 
	
	In this section, the maximum likelihood estimation problem of the association parameter $\theta$ of Gumbel's bivariate exponential distribution GBED($\theta$) is explained for a finite set of independent data and the ML-degree of the association parameter $\theta$ in GBED($\theta$) is defined.
	
	\textbf{Maximum likelihood estimation of the association parameter $\theta$}
	
	Let $X_{1}=(x_1,y_1)^\top, X_{2}=(x_2,y_2)^\top, \ldots, X_{n}=(x_n,y_n)^\top$ be random sample from Gumbel's bivariate exponential distribution GBED($\theta$), then the maximum likelihood estimator of $\theta$ is that value of $\theta \in [0,1]$, which maximizes the likelihood function given the data if it exists.
	
	The likelihood function for $\theta$ is
	\begin{align*}
		L(\theta |X_1,X_2,\ldots X_n) &=\prod_{i=1}^{n} f(x_{i},y_{i})\\
		&= \prod_{i=1}^{n} \left(  e^{-(x_{i}+y_{i}+\theta x_{i} y_{i})} [(1+\theta x_{i})(1+\theta y_{i})-\theta] \right), 
	\end{align*}
	
	and the log-likelihood function (up to an additive constant) is
	
	\begin{equation}
		\label{eq:1}
		\ell(\theta)= -\sum_{i=1}^{n} \left( \theta x_{i}y_{i} \right)+ \sum_{i=1}^{n} \log[1+(x_{i}+y_{i}-1)\theta + x_{i}y_{i} \theta^2].
	\end{equation}
	
	The score equation for maximizing $\ell(\theta)$ with respect to $\theta$ is 
	
	\begin{equation*}
		\sum_{i=1}^{n}\frac{(x_{i}+y_{i}-1)+2 x_{i} y_{i}\theta}{1+(x_{i}+y_{i}-1)\theta + x_{i} y_{i} \theta^2} =\sum_{i=1}^{n}x_{i}y_{i}, 
	\end{equation*}
	
	or
	\begin{equation}
		\label{eq:2}
		\sum_{i=1}^{n}\frac{(x_{i}y_{i})^2 \theta^2 +x_{i}y_{i}(x_{i}+y_{i}-1-2) \theta+ [x_{i} y_{i}-(x_{i}+y_{i}-1)]} {x_{i} y_{i} \theta^2+(x_{i}+y_{i}-1)\theta+1}=0.		
	\end{equation}
	
	This equation is a summation of rational functions in $\theta$, which will have more than one solution and does not have a closed-form solution, this makes it necessary to apply some computational algebraic techniques to solve this.
	
	Since $\mathbb{R}$ is not an algebraically closed field, the solutions of the score equation are considered over the complex field $\mathbb{C}$.
	
	\begin{note}
		The data considered throughout this paper excludes the data of the form $X_{i}=(x_{i},y_{i})^\top, X_{j}=(y_{i},x_{i})^\top$ and $X_{i}=X_{j}$ $\forall i,j, i\ne j$, as this data structure does not provide any helpful information.
	\end{note}
	
	\begin{definition}[\textbf{Maximum likelihood degree}]
		\label{def:1}
		The maximum likelihood degree or ML-degree of the association parameter $\theta$ of this model is the number of solutions of the score equation \eqref{eq:2}, counted with multiplicity over the complex field, given the data.
	\end{definition}
	
	Let for every $i=1,2,\ldots n$, $c_{i}=x_{i}y_{i}$, $d_{i}=x_{i}+y_{i}-1$, and
	\begin{equation}\label{eq:3}
		f_{i}(\theta) =c_{i}^2  \theta^2+c_{i}(d_{i}-2) \theta +(c_{i}-d_{i}), 
	\end{equation}
	\begin{equation}\label{eq:4}
		g_{i}(\theta)=c_{i}  \theta^2 +d_{i} \theta+1. 
	\end{equation}
	Then, equation \eqref{eq:2} can be rewritten as
	
	\[\sum_{i=1}^{n}\frac{c_{i}^2  \theta^2+c_{i}(d_{i}-2) \theta +(c_{i}-d_{i})}{c_{i}  \theta^2 +d_{i} \theta+1}=0,\]
	or
	\[\sum_{i=1}^{n} \frac{f_{i}(\theta)}{g_{i}(\theta)}=0,\]	
	or
	\[ \frac{f(\theta)}{g(\theta)}=0 \implies f(\theta)=0,\]
	with 
	\begin{equation}
		\label{eq:5}
		f(\theta) =	\displaystyle \sum_{i=1}^{n} \left(f_{i}(\theta) \displaystyle \prod_{j=1	j\ne i}^{n} g_{j}(\theta)\right),
	\end{equation}
	and
	\begin{equation}
		\label{eq:6}
		g(\theta)=\displaystyle \prod_{i=1	}^{n} g_{i}(\theta).
	\end{equation} 
	
	The degree of both polynomials $f(\theta)$ and $g(\theta)$ are $2n$.
	
	The solutions of the score equation are the zeros of $f(\theta)$. However, these solutions may contain the points where the score equation is not defined due to the cleared denominator. Therefore, the solutions of the score equation are the zeros of $f(\theta)$, which are not the zeros of $g(\theta)$. So, for the ML-degree, the common zeros of $f(\theta)$ and $g(\theta)$ should be removed from the zeros of $f(\theta)$.
	
	To determine the common zeros of $f(\theta)$ and $g(\theta)$ for the given data, the investigation of the geometry of the score equation is required. In the next section, the geometry of the score equation is discussed.

	\section{Geometry of  the score equation}
	Since $f(\theta)$ is a polynomial of degree $2n$, it will have $2n$ zeros in the complex field, counted with multiplicity \cite{CLOS}. The solutions of the score equation are in the variety of $f(\theta)$ (referred to as $V(f)$). Hence, the ML-degree of $\theta \le 2n$. For the ML-degree of $\theta$, the points of concern are 
	\[V(f)\setminus \left(V(f)\cap V(g)\right) =V(f)\setminus V(f,g).\]
	
	In this section, some results are obtained for the variety $V(f,g)$ to be non-empty, which are used further in the next section.

	\begin{theorem}
		\label{thm:2}
		$f(\theta)$ and $g(\theta)$ will have common zeros if and only if either $f_{k}(\theta)$ and $g_{k}(\theta)$ have common zeros for some $k\in \{1,2,\ldots,n\}$ or there is a pair $(g_{j},g_{k})$ $j\ne k$ having common zeros.
	\end{theorem}
	
	\begin{proof}
		The statement of this theorem can be rewritten in technical terms as:
		The variety $ V\left(f,g \right) \ne \emptyset$ if and only if either  $V(f_{k}, g_{k})\ne \emptyset$ for some $k \in \{1,2,\ldots n\}$ or there exists $j \ne k \in \{1,2, \ldots n\}$ such that $V(g_{j}, g_{k})\ne \emptyset$.
		
		Suppose $V\left(f,g\right) \ne \emptyset $, that is $ f(\theta)$ and $g(\theta)$ have a common zero, say $\alpha $.
		$f(\alpha)= 0$ and $g(\alpha)=0$.
		Consider
		\[g(\alpha)= \displaystyle \prod_{i=1}^{n} g_{i}(\alpha) =0,\]
		then there exists some $k \in \{1,2,\ldots n\}$ such that $g_{k}(\alpha)=0.$ 
		
		Now, 
		\[f(\alpha)=  \displaystyle \sum_{i=1}^{n} \left(f_{i}(\alpha) \displaystyle \prod_{j=1	j\ne i}^{n} g_{j}(\alpha)\right) =f_{k}(\alpha)\displaystyle \prod_{j=1	j\ne k}^{n} g_{j}(\alpha),\]
		since $g_{k}(\alpha)=0$.
		
		Therefore, $f(\alpha)=0 \implies$ either $f_{k}(\alpha)=0$ or $g_{j}(\alpha)=0, j\ne k.$
		
		Hence, either  $V(f_{k}, g_{k})\ne \emptyset$ or $V(g_{j}, g_{k})\ne \emptyset$, $j\ne k$.

		Conversely,
		first suppose $V(f_{k}, g_{k})\ne \emptyset$ for some $k \in \{1,2,\ldots n\}$ and say $\alpha_{1}$ is a common zero of $f_{k}(\theta)$ and $g_{k}(\theta)$, that is, $f_{k}(\alpha_{1})=0$ and $g_{k}(\alpha_{1})=0$, then,
		\[g(\alpha_{1})=\displaystyle \prod_{i=1}^{n} g_{i}(\alpha_{1}) =0,\]
		and
		\[f(\alpha_{1})=\displaystyle \sum_{i=1}^{n} \left(f_{i}(\alpha_{1}) \displaystyle \prod_{j=1	j\ne i}^{n} g_{j}(\alpha_{1})\right) =f_{k}(\alpha_{1})\displaystyle \prod_{j=1	j\ne k}^{n} g_{j}(\alpha_{1}) =0.\]

		Therefore, $\alpha_{1}\in V(f,g) \implies V(f,g)\ne \emptyset$. 
		
		Next , suppose $V(g_{j}, g_{k})\ne \emptyset$ for some $j\ne k$ and say $\alpha_{2}$ is a common zero of  $g_{j}(\theta)$ and $g_{k}(\theta)$. Then,
		\[g(\alpha_{2})=\displaystyle \prod_{i=1}^{n} g_{i}(\alpha_{2}) =0,\]
		and
		\begin{align*}
			f(\alpha_{2})&=\displaystyle \sum_{i=1}^{n} \left(f_{i}(\alpha_{2}) \displaystyle \prod_{t=1	t\ne i}^{n} g_{t}(\alpha_{2})\right) \\
			&= \left(f_{j}(\alpha_{2})\displaystyle \prod_{t=1	t\ne j}^{n} g_{t}(\alpha_{2}) \right) + \left(f_{k}(\alpha_{2})\displaystyle \prod_{t=1	t\ne k}^{n} g_{t}(\alpha_{2})\right) =0.
		\end{align*}
		
		Hence, $\alpha_{2}\in V(f,g)$, so $V(f,g)\ne \emptyset$.
	\end{proof}
	
	\begin{corollary}
		\label{thm:3}
		If all $g_{i}(\theta)$ have distinct zeros, then $f(\theta)$ and $g(\theta)$ will have common zeros if and only if there exists some $k$ such that $f_{k}(\theta)$ and $g_{k}(\theta)$ have common zero.
	\end{corollary}
	
	\begin{proof}
		The proof of this corollary follows from \hyperref[thm:2]{Theorem \ref{thm:2}}.
	\end{proof}
	
	\begin{lemma}
		\label{thm:4}
		$f_{i}(\theta)$ and $g_{i}(\theta)$ have common zeros if and only if $g_{i}(\theta)$ has double zero, and $V(f_{i}, g_{i})$ will be a singleton set.
	\end{lemma}

	\begin{proof}
		From \eqref{eq:3} and \eqref{eq:4}, $f_{i}(\theta)$ and $g_{i}(\theta)$ satisfy the following relation 
		\begin{equation}
			\label{eq:7}
			f_{i}(\theta)= c_{i}g_{i}(\theta)-2c_{i} \theta-d_{i}.
		\end{equation}
		Suppose that $f_{i}(\theta)$ and $g_{i}(\theta)$ have a common zero, say $\alpha$, then using relation \eqref{eq:7} and \eqref{eq:4},
		$\alpha=-\frac{d_{i}}{2c_{i}}$, and ${d_{i}}^2-4c_{i}=0.$
		
		Thus, $g_{i}(\theta)$ has double zero, that is, 
		\[V(g_{i})=\left\{\frac{1-x_{i}-y_{i}}{2x_{i} y_{i}},\frac{1-x_{i}-y_{i}}{2x_{i} y_{i}} \right\} \text{and}  \hspace{.9mm} V(f_{i})=\left\{ \frac{1-x_{i}-y_{i}}{2x_{i} y_{i}}, \frac{3-x_{i}-y_{i}}{2x_{i} y_{i}} \right\}. \] Therefore, \[V(f_{i},g_{i})=\left\{ -\frac{d_{i}}{2c_{i}} \right\}= \left\{\frac{1-x_{i}-y_{i}}{2x_{i} y_{i}} \right\}.\]
		
		Conversely, suppose $g_{i}(\theta)$ has double zero, by \eqref{eq:4} $V(g_{i})=\left\{-\frac{d_{i}}{2c_{i}},-\frac{d_{i}}{2c_{i}}\right\}$ and by relation \eqref{eq:7}, 
		\begin{equation*}
			f_{i} \left(-\frac{d_{i}}{2c_{i}} \right) = - 2c_{i} \left(-\frac{d_{i}}{2c_{i}} \right)-d_{i}=0,
		\end{equation*}
		and \[V(f_{i})=\left\{ -\frac{d_{i}}{2c_{i}}, \frac{4-d_{i}}{2c_{i}} \right\}.\]
		
		Hence, $V(f_{i}, g_{i})$ is a singleton set.
	\end{proof}

	\begin{lemma}
		\label{thm:5}
		At most one $g_{k}(\theta)$ will have double zero, if all pairs $(g_{i},g_{j})$ $i\ne j$ have common zeros.
	\end{lemma}
	
	\begin{proof}
		Suppose $g_{k}(\theta)$ and $g_{l}(\theta)$ $k\ne l$ both have double zero, and it is given that all pairs $(g_{i},g_{j})$ $i\ne j$ have common zeros, then all zeros of $g_{k}(\theta)$ and $g_{l}(\theta)$ are same, which is not valid for the given data.
		
		Hence, by contradiction, the statement is true.
	\end{proof}

	\begin{note}
		For the given data, $V(g_{i},g_{j})$ is either empty or a singleton set for any $i\ne j \in \{1,2,\ldots,n\}$.
	\end{note}
	
	Let $g_{1}(\theta), g_{2}(\theta), g_{3}(\theta)$ be three polynomials as defined in \eqref{eq:4} and $V(g_{1})=\{ \alpha, \alpha_{1}\}$, $V(g_{2})=\{ \alpha, \alpha_{2}\}$, $V(g_{3})=\{ \alpha_{1},\alpha_{2}\}$, then $V(g_{i},g_{j})\ne \emptyset$ $\forall i\ne j \in \{1,2,3\}$ but $V(g_{1},g_{2},g_{3})= \emptyset$.	This phenomenon is not true for $n\ge 4$.

	\begin{theorem}
		\label{thm:6}
		For $n\ge 4$, $V(g_{1},g_{2},\ldots,g_{n})$ is non-empty and a singleton set, if every pair $(g_{i},g_{j})$ $i\ne j$ have common zeros.
	\end{theorem}
	
	\begin{proof} 
		
		When $n= 4$ and $V(g_{i},g_{j})\ne \emptyset$ $\forall i\ne j \in \{1,2,3,4\}$, then $V(g_{i},g_{j})$ will be a singleton set for every $i\ne j \in \{1,2,3,4\}$.
		
		Assume as above $V(g_{1})=\{ \alpha, \alpha_{1}\}$, $V(g_{2})=\{ \alpha, \alpha_{2}\}$, $V(g_{3})=\{ \alpha_{1},\alpha_{2}\}$. Now, let $V(g_{4})=\{\beta_{1},\beta_{2}\}$.
		
		Since $V(g_{1},g_{4})\ne \emptyset$ $\implies$ $ \beta_{1} = \alpha$ or $\alpha_{1}$.
        
		    \case {$ \beta_{1} = \alpha$,} then $V(g_{1},g_{4})\ne \emptyset$ and $V(g_{2},g_{4})\ne \emptyset$, but $V(g_{3},g_{4})= \emptyset$.
			\item \case {$ \beta_{1} = \alpha_{1}$,} then $V(g_{1},g_{4})\ne \emptyset$ and $V(g_{3},g_{4})\ne \emptyset$, but $V(g_{2},g_{4})= \emptyset$.

		Thus, $\alpha \in V(g_{3})$ and $ \beta_{1} = \alpha$ will only make $V(g_{i},g_{j})\ne \emptyset$ $\forall i\ne j \in \{1,2,3,4\}$.
		
		Therefore, $V(g_{i})=\{ \alpha,\alpha_{i}\}$, $\forall i\in \{1,2,3,4\}$ with $\alpha$ and all $\alpha_{i}$'s being distinct and $V(g_{1},g_{2},g_{3},g_{4}) =\{\alpha \}$.
		
		Next, assume this statement is true for any fixed $m\ge4$, that is if $V(g_{i},g_{j})\ne \emptyset$ $\forall i\ne j\in \{1,2,\ldots,m\}$, then $V(g_{1},g_{2},\ldots,g_{m}) =\{\alpha\}$. Now, if one more polynomial $g_{m+1}(\theta)$ is added such that $V(g_{i},g_{j})\ne \emptyset$ $\forall i\ne j \in \{1,2,\ldots,m,m+1\}$, then 
		\begin{align*}
			V(g_{1},g_{2},\ldots,g_{m},g_{m+1}) &= V(g_{1},g_{2},\ldots,g_{m})\cap V(g_{m+1})\\
			&=\{\alpha\} \cap \{\beta_{1},\beta_{2}\}= \emptyset.
		\end{align*}
		If $\alpha$ is distinct from $\beta_{1}$ and $\beta_{2}$, then $g_{m+1}(\theta)$ can have common zeros with at most two polynomials only, since $\beta_{1} \ne \beta_{2}$ are only two distinct zeros, which contradicts the hypothesis that $V(g_{i},g_{m+1}) \ne \emptyset$ $\forall i \in \{1,2,\ldots,m\}$. 
		
		Therefore, $V(g_{1},g_{2},\ldots,g_{m},g_{m+1}) \ne \emptyset$ and $V(g_{1},g_{2},\ldots,g_{m},g_{m+1}) =\{\alpha\}$.
		
		Hence, this result is true for all $n\ge4$.
	\end{proof}

	\begin{theorem}
		\label{thm:7}
		$V(g_{1},g_{2},\ldots,g_{n})$ is non-empty and a singleton set, if all pairs $(g_{i},g_{j})$ $i\ne j$ have common zeros and one $g_{k}(\theta)$ has double zero.
	\end{theorem}
	\begin{proof}
		Suppose $\alpha$ be a double zero of $g_{k}(\theta)$ and all pairs $(g_{i},g_{j})$ $i\ne j$ have common zeros, then this $\alpha$ will be zero of every $g_{i}(\theta)$ $i\ne k$, hence $\alpha$ is a zero of every $g_{i}(\theta)$ and none of the other $g_{i}(\theta)$ can have double zero, for the given data.
		
		Hence, $V(g_{1},g_{2},\ldots,g_{n})=\{\alpha\} \ne \emptyset$.
	\end{proof}
	
	In this result, there is no restriction on sample size compared to \hyperref[thm:6]{Theorem \ref{thm:6}}, but has one more condition of double zero of $g_{k}(\theta)$.

	\section{The multiplicity of common zeros of \texorpdfstring{$f(\theta)$}{} and \texorpdfstring{$g(\theta)$}{}}
	As mentioned in section 3, the ML-degree of the association parameter $\theta$ of GBED($\theta$) will be the number of elements in $V(f)\setminus V(f,g)$, counted with multiplicity. Since $f(\theta)$ has $2n$ zeros in $\mathbb{C}$ (counted with multiplicity), thus $V(f,g)$ is of concern for the calculation of the ML-degree. In the previous section, all the possibilities for $V(f,g)\ne \emptyset$ have been discussed. In this section, the multiplicity of every element of $V(f,g)$ is counted in $f(\theta)$.

	\begin{lemma}
		\label{thm:8}
		The multiplicity of a common zero $($say $\alpha)$ of $f(\theta)$ and $g(\theta)$ in $f(\theta)$ is one, if $\alpha$ is a double zero of $g_{k}(\theta)$ and all $g_{i}(\theta)$ have distinct zeros.
	\end{lemma}
	
	\begin{proof}
		Given that $\alpha$ is a double zero of $g_{k}(\theta)$, by \hyperref[thm:4]{Lemma \ref{thm:4}}, $V(f_{k},g_{k})=\{\alpha\}$, hence by \hyperref[thm:2]{Theorem \ref{thm:2}}, $V(f,g)\ne \emptyset$ and $\alpha \in V(f,g)$. 
		
		Next, the multiplicity of $\alpha$ is counted in $f(\theta)$.
		
		Suppose $\beta_{k}$ is other zero of $f_{k}(\theta)$, then 
		\begin{align*}
			f(\theta)& =\displaystyle \sum_{i=1}^{n} \left(f_{i}(\theta) \displaystyle \prod_{j=1	j\ne i}^{n} g_{j}(\theta)\right)\\
			&=	f_{k}(\theta) \left(\displaystyle \prod_{j=1 j\ne k}^{n} g_{j}(\theta)\right)+ \displaystyle \sum_{i=1 i\ne k}^{n} \left(f_{i}(\theta) \displaystyle \prod_{j=1	j\ne i}^{n} g_{j}(\theta)\right),
		\end{align*}
		or
		\begin{multline*}
			f(\theta)= (\theta -\alpha) (\theta-\beta_{k}) \left( \prod_{j=1,j\ne k}^{n} g_{j}(\theta) \right)\\
			+ (\theta-\alpha)^2 \left(\sum_{i=1,i\ne k}^{n} \left(f_{i}(\theta)\prod_{j=1,j\ne i,k}^{n} g_{j}(\theta)\right) \right), 
		\end{multline*}
		or
		\begin{equation*}
			f(\theta)= (\theta-\alpha) h(\theta),
		\end{equation*}	              
		where 
		\begin{equation*}
			h(\theta)=(\theta-\beta_{k}) \left( \prod_{j=1,j\ne k}^{n} g_{j}(\theta) \right) + (\theta-\alpha) \left(\sum_{i=1,i\ne k}^{n} \left(f_{i}(\theta)\prod_{j=1,j\ne i,k}^{n} g_{j}(\theta)\right) \right)
		\end{equation*}
		does not have $\alpha $ as a zero, since $V(g_{i},g_{k})= \emptyset$ $\forall i\ne k$.
		
		Hence, multiplicity of $\alpha$ in $f(\theta)$ is one.
	\end{proof}

	\begin{lemma}
		\label{thm:9}
		The multiplicity of a common zero $($say $\alpha)$ of $f(\theta)$ and $g(\theta)$ in $f(\theta)$ is $n_{1}-1$, if none of the $g_{i}(\theta)$ have double zeros and $\alpha$ is a common zero of exactly $n_{1}$ $g_{i}(\theta)$'s $(2\le n_{1} \le n)$.
	\end{lemma}
	
	\begin{proof}
		Suppose $\alpha$ is a common zero of exactly $n_{1}$ $g_{i}(\theta)$, say \[g_{1}(\theta), g_{2}(\theta), \ldots, g_{n_{1}}(\theta).\] Since $n_{1} \ge 2 $, at least one pair $(g_{i},g_{j})$ have a common zero, then by \hyperref[thm:2]{Theorem \ref{thm:2}}, $V(f,g)\ne \emptyset$ and $\alpha \in V(f,g)$. To count the multiplicity of $\alpha$ in $f(\theta)$, suppose $\alpha_{i}$ ($\alpha \ne \alpha_{i}$) are the other zeros of $g_{i}(\theta)$ $\forall i\in \{1,2,\ldots, n_{1}\}$, then
		\begin{align*}
			f(\theta)& =\displaystyle \sum_{i=1}^{n} \left(f_{i}(\theta) \displaystyle \prod_{j=1	j\ne i}^{n} g_{j}(\theta)\right)\\
			&=\displaystyle \sum_{i=1}^{n_{1}} \left(f_{i}(\theta) \displaystyle \prod_{j=1	j\ne i}^{n} g_{j}(\theta)\right)+\displaystyle \sum_{i=n_{1}}^{n} \left(f_{i}(\theta) \displaystyle \prod_{j=1	j\ne i}^{n} g_{j}(\theta)\right),
		\end{align*}
		or
		\begin{multline*}
			f(\theta)=	\displaystyle \sum_{i=1}^{n_{1}} \left(f_{i}(\theta) \left(\displaystyle \prod_{j=1	j\ne i}^{n_{1}} g_{j}(\theta)\right) \left(\displaystyle \prod_{j=n_{1}}^{n} g_{j}(\theta)\right) \right)\\
			+\displaystyle \sum_{i=n_{1}}^{n} \left( f_{i}(\theta) \left( \displaystyle \prod_{j=1}^{n_{1}} g_{j}(\theta)\right) \left(\displaystyle \prod_{j=n_{1}	j\ne i}^{n} g_{j}(\theta)\right)\right),
		\end{multline*}
		or
		\begin{multline*}
			f(\theta)= \left(\theta-\alpha \right)^{n_{1}-1} \left(\displaystyle \sum_{i=1}^{n_{1}} \left(f_{i}(\theta)  \left(\displaystyle \prod_{j=1 j\ne i}^{n_{1}}(\theta-\alpha_{j})\right) \left(\displaystyle \prod_{j=n_{1}}^{n} g_{j}(\theta)\right) \right) \right)\\
			+  \left(\theta-\alpha \right)^{n_{1}}  \left(\displaystyle \sum_{i=n_{1}}^{n} \left(f_{i}(\theta) \left(\displaystyle \prod_{j=1}^{n_{1}}(\theta-\alpha_{j})\right) \left(\displaystyle \prod_{j=n_{1}, j\ne i}^{n} g_{j}(\theta)\right) \right) \right),
		\end{multline*}
		or
		\begin{equation*}
			f(\theta)= (\theta-\alpha)^{n_{1}-1} h(\theta),
		\end{equation*}
		where
		\begin{multline*}
			h(\theta)=\displaystyle \sum_{i=1}^{n_{1}} \left(f_{i}(\theta) \left( \displaystyle \prod_{j=1 j\ne i}^{n_{1}}(\theta-\alpha_{j})\right) \left(\displaystyle \prod_{j=n_{1}}^{n} g_{j}(\theta)\right) \right)\\
			+  \left(\theta-\alpha \right)  \left(\displaystyle \sum_{i=n_{1}}^{n} \left(f_{i}(\theta) \left(\displaystyle \prod_{j=1}^{n_{1}}(\theta-\alpha_{j}) \right) \left(\displaystyle \prod_{j=n_{1}, j\ne i}^{n} g_{j}(\theta) \right) \right) \right),
		\end{multline*}
		does not have $\alpha $ as a zero, since $V(f_{i},g_{i})=\emptyset$ $\forall i \in \{1,2,\ldots,n_{1}\}$ and $\alpha$ is not a zero of any $g_{i}(\theta)$ other than $n_{1}$ $g_{i}(\theta)$.
		
		Hence, the multiplicity of $\alpha$ in $f(\theta)$ is $n_{1}-1$.
	\end{proof}

	\begin{lemma}
		\label{thm:10}
		The multiplicity of a common zero $($say $\alpha)$ of $f(\theta)$ and $g(\theta)$ in $f(\theta)$ is $n_{1}$, if $\alpha$ is a common zero of exactly $n_{1}$ $g_{i}(\theta)$'s with one of the $g_{i}(\theta)$ having double zero $(2\le n_{1} \le n)$.
	\end{lemma}
	
	\begin{proof}
		Suppose $\alpha$ is a double zero of $g_{1}(\theta)$ and $\alpha$ is a common zero of exactly $n_{1}$ $g_{i}(\theta)$ (say $g_{1}(\theta), g_{2}(\theta), \ldots, g_{n_{1}}(\theta)$), then  by \hyperref[thm:4]{Lemma \ref{thm:4}}, $V(g_{1})=\{\alpha,\alpha\}$ and $V(f_{1},g_{1})=\{\alpha\}$. Thus, by \hyperref[thm:2]{Theorem \ref{thm:2}}, $V(f,g)\ne \emptyset$ and $\alpha \in V(f,g)$. To count the multiplicity of $\alpha$ in $f(\theta)$, suppose $\beta$ and $\alpha_{i}$ ($\beta \ne \alpha_{i} \ne \alpha$) are the other zeros of $f_{1}(\theta)$ and $g_{i}(\theta)$ $\forall i \in \{2,3,\ldots,n_{1}\}$, respectively, then	
		\begin{align*}
			f(\theta)& =\displaystyle \sum_{i=1}^{n} \left(f_{i}(\theta) \displaystyle \prod_{j=1	j\ne i}^{n} g_{j}(\theta)\right)\\
			&= f_{1}(\theta) \left(\displaystyle \prod_{j=2}^{n} g_{j}(\theta) \right) + \displaystyle \sum_{i=2}^{n_{1}} \left(f_{i}(\theta) \displaystyle \prod_{j=1	j\ne i}^{n} g_{j}(\theta)\right)+\displaystyle \sum_{i=n_{1}}^{n} \left(f_{i}(\theta) \displaystyle \prod_{j=1	j\ne i}^{n} g_{j}(\theta)\right),
		\end{align*}
		or
		\begin{multline*}
			f(\theta)= f_{1}(\theta) \left(\displaystyle \prod_{j=2}^{n_{1}} g_{j}(\theta) \right)  \left(\displaystyle \prod_{j=n_{1}}^{n} g_{j}(\theta) \right) \\
			+\displaystyle \sum_{i=2}^{n_{1}} \left(f_{i}(\theta) \left(\displaystyle \prod_{j=1	j\ne i}^{n_{1}} g_{j}(\theta) \right) \left(\displaystyle \prod_{j=n_{1}}^{n} g_{j}(\theta)\right) \right) \\
			+\displaystyle \sum_{i=n_{1}}^{n} \left( f_{i}(\theta) \left(\displaystyle \prod_{j=1}^{n_{1}} g_{j}(\theta)\right) \left(\displaystyle \prod_{j=n_{1}	j\ne i}^{n} g_{j}(\theta)\right) \right),
		\end{multline*}
		or 
		\begin{multline*}
			f(\theta)=\left( \theta -\beta \right) \left(\theta-\alpha \right)^{n_{1}}  \left(\displaystyle \prod_{j=2}^{n_{1}} \left(\theta-\alpha_{j}\right) \right)  \left(\displaystyle \prod_{j=n_{1}}^{n} g_{j}(\theta) \right) \\
			+ \left(\theta-\alpha \right)^{n_{1}} \left(\displaystyle \sum_{i=2}^{n_{1}} \left(f_{i}(\theta)  \left(\displaystyle \prod_{j=2 j\ne i}^{n_{1}}(\theta-\alpha_{j})\right) \left(\displaystyle \prod_{j=n_{1}}^{n} g_{j}(\theta)\right) \right) \right)\\
			+  \left(\theta-\alpha \right)^{n_{1}+1}  \left(\displaystyle \sum_{i=n_{1}}^{n} \left(f_{i}(\theta) \left(\displaystyle \prod_{j=2}^{n_{1}}(\theta-\alpha_{j})\right) \left(\displaystyle \prod_{j=n_{1}, j\ne i}^{n} g_{j}(\theta)\right) \right) \right),
		\end{multline*}
		or
		\begin{equation*}
			f(\theta)= (\theta-\alpha)^{n_{1}} h(\theta),
		\end{equation*}
		where
		\begin{multline*}
			h(\theta)=\left( \theta -\beta \right)  \left(\displaystyle \prod_{j=2}^{n_{1}} \left(\theta-\alpha_{j}\right) \right)  \left(\displaystyle \prod_{j=n_{1}}^{n} g_{j}(\theta) \right)\\
			+ \displaystyle \sum_{i=2}^{n_{1}} \left(f_{i}(\theta) \left( \displaystyle \prod_{j=2 j\ne i}^{n_{1}}(\theta-\alpha_{j})\right) \left(\displaystyle \prod_{j=n_{1}}^{n} g_{j}(\theta)\right) \right)\\
			+  \left(\theta-\alpha \right)  \left(\displaystyle \sum_{i=n_{1}}^{n} \left(f_{i}(\theta) \left(\displaystyle \prod_{j=2}^{n_{1}}(\theta-\alpha_{j})\right) \left(\displaystyle \prod_{j=n_{1}, j\ne i}^{n} g_{j}(\theta)\right) \right) \right),
		\end{multline*}
		does not have $\alpha $ as a zero, since $\alpha$ is not a zero of any $f_{i}(\theta)$ and $g_{j}(\theta)$ other than $f_{1}(\theta)$ and $n_{1}$ $g_{i}(\theta)$.
		
		Hence, the multiplicity of $\alpha$ in $f(\theta)$ is $n_{1}$.
	\end{proof}

	\section{The ML degree of the association parameter}
	In this section, the ML-degree of the association parameter $\theta$ is calculated for all cases where $V(f,g)$ is non-empty, as discussed in section 3. This calculation uses the multiplicity of elements of $V(f,g)$ in $f(\theta)$ counted in section 4.

	\begin{theorem}
		\label{thm:11}
		For $n\ge4$, the ML-degree of the association parameter $\theta$ in GBED($\theta$) is $n+1$, if all pairs $(g_{i},g_{j})$ $i \ne j$ have common zeros and none of $g_{i}(\theta)$ have double zeros.
	\end{theorem}
	
	\begin{proof}
		Given that $V(g_{i},g_{j})\ne \emptyset$ $\forall i\ne j$, then by \hyperref[thm:6]{Theorem \ref{thm:6}}, $V(g_{1},g_{2},\ldots,g_{n})$ is a singleton set, say $V(g_{1},g_{2},\ldots,g_{n})=\{\alpha\}$, and it is given that none of  $g_{i}(\theta)$ has double zero, that is, $V(f_{i},g_{i})= \emptyset$ $\forall i$. Thus, by \hyperref[thm:2]{Theorem \ref{thm:2}} $V(f,g)\ne \emptyset$ and $\alpha$ is the only common zero of $f(\theta)$ and $g(\theta)$, and by \hyperref[thm:9]{Lemma \ref{thm:9}}, the multiplicity of $\alpha$ in $f(\theta)$ is $n-1$.
		
		Hence, the ML-degree of the association parameter $\theta$ in GBED($\theta$) will be the number of elements in $V(f)\setminus V(f,g)$ (counted with multiplicity), which is equal to $2n-(n-1)=n+1$.
	\end{proof}

	For $n=3$, let $g_{1}(\theta)$, $g_{2}(\theta)$, and $g_{3}(\theta)$ be three polynomials as defined in \eqref{eq:4}. If all pairs $(g_{i},g_{j})$ $i\ne j \in \{1,2,3\}$ have common zeros and none of the $g_{i}(\theta)$ have double zero, then there are two possibilities for $V(g_{1},g_{2},g_{3})$: either $V(g_{1},g_{2},g_{3})\ne \emptyset$ or $V(g_{1},g_{2},g_{3})=\emptyset$.
	\begin{enumerate}
		\item 	If $V(g_{1},g_{2},g_{3})\ne \emptyset$, then the varieties are of the form: $V(g_{1})=\{ \alpha, \alpha_{1}\}$, $V(g_{2})=\{ \alpha, \alpha_{2}\}$, $V(g_{3})=\{ \alpha,\alpha_{3}\}$ ($\alpha\ne \alpha_{1}\ne \alpha_{2}\ne \alpha_{3}$) and $\alpha \in V(g_{1},g_{2},g_{3})$. Then, by \hyperref[thm:3]{Corollary \ref{thm:3}} and \hyperref[thm:9]{Lemma \ref{thm:9}}, $\alpha$ is the only common zero of $f(\theta)$ and $g(\theta)$ with multiplicity $2$ in $f(\theta)$. Hence, the ML-degree of the association parameter $\theta$ is $6-2=4$. 
		\item If $V(g_{1},g_{2},g_{3})=\emptyset$, then the varieties are of the form : $V(g_{1})=\{ \alpha, \alpha_{1}\}$, $V(g_{2})=\{ \alpha, \alpha_{2}\}$, $V(g_{3})=\{\alpha_{1}, \alpha_{2}\}$, $\alpha\ne \alpha_{1}\ne \alpha_{2}$. Then, by \hyperref[thm:2]{Theorem \ref{thm:2}}, $V(f,g)=\{\alpha, \alpha_{1}, \alpha_{2}\}$, and it is given that $V(f_{i},g_{i})= \emptyset$ for every $i$, then by \hyperref[thm:9]{Lemma \ref{thm:9}}, the multiplicity of each of $\alpha, \alpha_{1}, \alpha_{2}$ in $f(\theta)$ is $1$. Hence, the ML-degree of the association parameter $\theta$ is $6-3=3$.
	\end{enumerate}

	\begin{theorem}
		\label{thm:12}
		The ML-degree of the association parameter $\theta$ in GBED($\theta$) is $n$, if all pairs $(g_{i},g_{j})$ $i \ne j$ have common zeros with one of $g_{i}(\theta)$ having double zero.
	\end{theorem}
	\begin{proof}
		Given that $V(g_{i},g_{j})\ne \emptyset$ $\forall i\ne j$ and one $g_{k}(\theta)$ has double zero, then by \hyperref[thm:7]{Theorem \ref{thm:7}}, $V(g_{1},g_{2},\ldots,g_{n})$ is a singleton set, say $V(g_{1},g_{2},\ldots,g_{n})=\{\alpha\}$, and by \hyperref[thm:2]{Theorem \ref{thm:2}}, $V(f,g)\ne \emptyset$ and $V(f,g)=\{\alpha\}$. Thus, by \hyperref[thm:10]{Lemma \ref{thm:10}}, the multiplicity of $\alpha$ in $f(\theta)$ is $n$.
		
		Hence, the ML-degree of the association parameter $\theta$ in GBED($\theta$) will be the number of elements in $V(f)\setminus V(f,g)$ (counted with multiplicity), which is equal to $2n-n=n$.
	\end{proof}

	\begin{theorem}
		\label{thm:13}
		The ML-degree of the association parameter $\theta$ in GBED($\theta$) is $2n-n_{1} $, if all $g_{i}(\theta)$ have distinct zeros and exactly $n_{1}$  $(\le n)$ $g_{i}(\theta)$'s have double zeros.
	\end{theorem}
	\begin{proof}
		Given that $V(g_{i},g_{j}) =\emptyset$ $\forall$ $i\ne j$, and $n_{1}$  $g_{i}(\theta)$ have double zeros, then by \hyperref[thm:4]{Lemma \ref{thm:4}} $V(f_{i},g_{i})\ne \emptyset $ $\forall$ $i\in \{1,2,\ldots,n_{1}\}$ and $\alpha_{i} \in V(f_{i},g_{i})$ $\forall$ $i\in \{1,2,\ldots,n_{1}\}$. Hence, by \hyperref[thm:3]{Corollary \ref{thm:3}} $V(f,g)=\{\alpha_{1},\alpha_{2},\ldots,\alpha_{n_1}\}$, and thus by \hyperref[thm:8]{Lemma \ref{thm:8}}, each $\alpha_{i}$ will be zero of $f(\theta)$ with multiplicity one.
		
		Hence, the ML-degree of the association parameter $\theta$ of GBED($\theta$) will be the number of elements in $V(f)\setminus V(f,g)$ (counted with multiplicity), which is equal to $2n-n_{1}.$
	\end{proof}

	\begin{theorem}
		\label{thm:14}
		The ML-degree of the association parameter $\theta$ in GBED($\theta$) is $2n+l-m$, if $g(\theta)$ has exactly $l$ repeated zeros, each with multiplicity $n_{k}(\ge 2)$ $\left(\sum_{k=1}^{l} n_{k} =m \le 2n\right)$ and none of the $g_{i}(\theta)$ have double zeros.
	\end{theorem}
	
	\begin{proof}
		Since none of the $g_{i}(\theta)$ have double zeros, so repeated zero of $g(\theta)$ will occur only if it is common with $n_{k}$ $g_{i}(\theta)$'s. Given that $g(\theta)$ has exactly $l$ repeated zeros, say $\alpha_{1}, \alpha_{2},\ldots,\alpha_{l}$ each with muliplicity $n_{1},n_{2},\ldots,n_{l}$, respectively. Suppose $\alpha_{k}$ is a common zero of $g_{k_1}(\theta),g_{k_2}(\theta),\ldots,g_{k_{n_{k}}}(\theta)$. Since $n_{k}\ge 2$, by \hyperref[thm:2]{Theorem \ref{thm:2}}, $V(f,g)\ne \emptyset$ and $\alpha_{k} \in V(f,g)$, and by \hyperref[thm:9]{Lemma \ref{thm:9}}, multiplicity of $\alpha_{k}$ in $f(\theta)$ is $n_{k}-1$, for all $k=1,2,\ldots,l$.
		
		Therefore, $V(f,g)=\{\alpha_{1},\alpha_{2},\ldots,\alpha_{l}\}$.
		
		Hence, the ML-degree of the association parameter $\theta$ of GBED($\theta$) will be the number of elements in $V(f)\setminus V(f,g)$ (counted with multiplicity), which is equal to 
		\[2n-\sum_{k=1}^{l}(n_{k}-1) = 2n-(m-l) =2n+l-m.\]
	\end{proof}

	\begin{theorem}
		\label{thm:15}
		The ML-degree of the association parameter $\theta$ in GBED($\theta$) is $2n-m$, if exactly $l$ $g_{i}(\theta)$'s has double zeros and each of the double zeros is a common zero of $n_{k}$ $(\ge 2)$ $g_{i}(\theta)$'s $\left(\sum_{k=1}^{l} n_{k} =m< 2n \right)$.
	\end{theorem}
	
	\begin{proof}
		Suppose $\alpha_{1}, \alpha_{2},\ldots,\alpha_{l}$ are double zeros of $g_{1_1}(\theta),g_{2_1}(\theta),\ldots,g_{l_1}(\theta)$, respectively, and $\alpha_{k}$ is common zero of $g_{k_1}(\theta),g_{k_2}(\theta),\ldots,g_{k_{n_{k}}}(\theta)$. Since  $\alpha_{k}$ is a double zero of $g_{k_1}(\theta)$, and $\alpha_{k} \in V(g_{k_1},g_{k_2},\ldots,g_{k_{n_{k}}})$, then by \hyperref[thm:2]{Theorem \ref{thm:2}} $\alpha_{k} \in V(f,g)$, and by \hyperref[thm:10]{Lemma \ref{thm:10}}, and the multiplicity of $\alpha_{k}$ in $f(\theta)$ is $n_{k}$, for all $k=1,2,\ldots,l$.
		
		Therefore, $V(f,g)=\{\alpha_{1},\alpha_{2},\ldots,\alpha_{l}\}$.
		
		Hence, the ML-degree of the association parameter $\theta$ of GBED($\theta$) will be the number of elements in $V(f)\setminus V(f,g)$ (counted with multiplicity), which is equal to 
		\[ 2n-\sum_{k=1}^{l}n_{k}= 2n-m.\]
	\end{proof}
	
	In the above theorem, every double zero is also a common zero of  $n_{k}$ $g_{j}(\theta)$, but each $g_{j}(\theta)$ is a polynomial of degree $2$ only. Hence, any $g_{j}(\theta)$ can be grouped with maximum two distinct double zeros. So, $\sum_{k=1}^{l} n_{k} =m\le  2n-l$. Moreover, since $g(\theta)= \prod_{i=1}^{n} g_{i}(\theta)$ has $2n$ zeros, and $l$ $g_{i}(\theta)$ have double zero, so the remaining zeros of $g(\theta)$ are $2(n-l)$. Since each double zero is also a zero of at least one more $g_{j}(\theta)$, which makes $l \le \frac{2n}{3}$.

Here are some examples that illustrate how the ML-degree of the association parameter varies.
\begin{example}
	 Let $n=3$ and the random sample
     \[X_{1}=(2.25,0.25)^\top,  X_{2}=(0.1,0.3)^\top,  X_{3}=(0.5,0.7)^\top.\] 
     Then 
	\[c_{1}=0.5625, c_{2}=0.03, c_{3}=0.35,\]
 \[d_{1}=1.5,d_{2}=-0.6,d_{3}=0.2.\]
	
	The score equation is 
	\[ \frac{f(\theta)}{g(\theta)}=0,\]
	where
	\begin{equation*}
		f(\theta)=0.0056 \theta^{6} -0.1287 \theta^{5} +0.3543 \theta^{4} +0.5602 \theta^{3}+ 0.3267 \theta^{2}+0.5917 \theta -0.1575,
	\end{equation*} 
	and 
		\begin{equation*}
		g(\theta)=0.0059 \theta^{6} -0.099\theta^{5} -0.1492 \theta^{4} -0.039 \theta^{3}+ 0.2225 \theta^{2}+1.1 \theta +1.
	\end{equation*} 
    $V(f_{1}, g_{1})=\{-1.3333\}$, and $g_{1}(\theta)$ has double zero, then by \hyperref[thm:2]{Theorem \ref{thm:2}}, $V(f,g)\ne \emptyset.$
    
Both polynomials$f(\theta)$ and $g(\theta)$ are of degree $6$, and 
\[V(f)=\{-1.3333, 0.2257, 4.8053, 19.6123, -0.0909\pm 0.9946 i\} \]
\[V(g)=\{-1.3333, 1.8350, 18.1650, -0.2857\pm 1.666i\},\]
and $V(f,g)=\{-1.3333\}$ with multiplicity $1$, and $-1.3333$ is double root of $g(\theta)$. Therefore, from \hyperref[thm:13]{Theorem \ref{thm:13}}, the ML-degree of the parameter $\theta$ is $5$. The likelihood function of $\theta$ given the random sample $X_{1},X_{2},X_{3}$ is
\begin{equation*}
	L(\theta\mid X_{1},X_{2},X_{3})= e^{1.8-(0.9425) \theta} g(\theta),
\end{equation*}
given as follow: 
	\begin{figure}[H]
		\centering
		\includegraphics[width=0.9\linewidth]{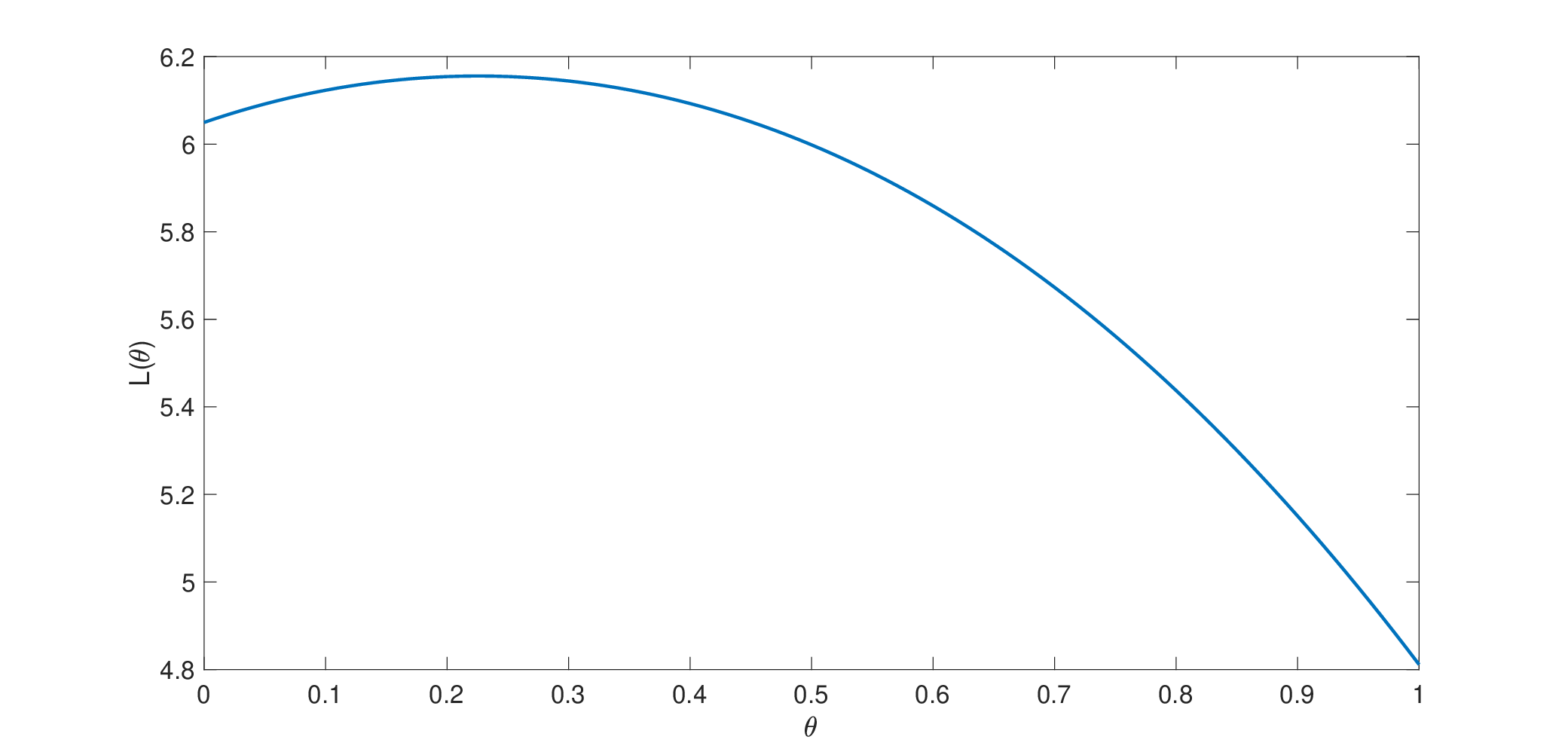}
		\caption{\footnotesize \textbf{\emph{The likelihood function of the parameter $\theta$ given the sample $X_{1},X_{2},X_{3}$.}}}
		\label{fig:ex1ch3}
	\end{figure}
	So, the MLE of the parameter $\theta$ is $0.2257$.

\end{example}

\begin{example}
	Let $n=3$ and the random sample 
	\[X_{1}=(0.406732,0.117077)^\top, X_{2}=(0.333333,0.166667)^\top,  X_{3}=(1,2)^\top,\]
	 then the score equation is 
	\[ \frac{f(\theta)}{g(\theta)}=0,\]
	where
	\begin{equation*}
		f(\theta)=0.0111 \theta^{6} -0.232\theta^{5} +1.7058 \theta^{4} -5.115 \theta^{3}+ 4.7782 \theta^{2}+1.3755\theta +1.0794,
	\end{equation*} 
	and 
	\begin{equation*}
		g(\theta)=0.0053 \theta^{6} -0.0952\theta^{5} +0.5847 \theta^{4} -1.3201 \theta^{3}+ 0.3889 \theta^{2}+1.0238 \theta +1.
	\end{equation*} 
    $V(g_{1},g_{2})=\{3\}$, from \hyperref[thm:2]{Theorem \ref{thm:2}}, $V(f,g)\ne \emptyset$.

	Both polynomials$f(\theta)$ and $g(\theta)$ are of degree $6$, and 
	\[V(f)=\{8.0235,6.3597,3.8258,3,-0.1781\pm 0.3659i\} \]
	\[V(g)=\{7,6,3,-0.5\pm 0.5i\},\]
	and $V(f,g)=\{3\}$ with multiplicity $1$, and $3$ is double root of $g(\theta)$. Therefore, from \hyperref[thm:14]{Theorem \ref{thm:14}}, the ML-degree of the parameter $\theta$ is $5$. The solutions of the score equation do not belong to the parameter space $[0,1]$. Thus, in that case, the likelihood function is either increasing or decreasing. The likelihood function of $\theta$ given the random sample $X_{1},X_{2},X_{3}$ is
	\begin{equation*}
		L(\theta\mid X_{1},X_{2},X_{3})= e^{1.9762-(2.1032) \theta} g(\theta),
	\end{equation*}
	given as follow: 
	\begin{figure}[H]
		\centering
		\includegraphics[width=0.9\linewidth]{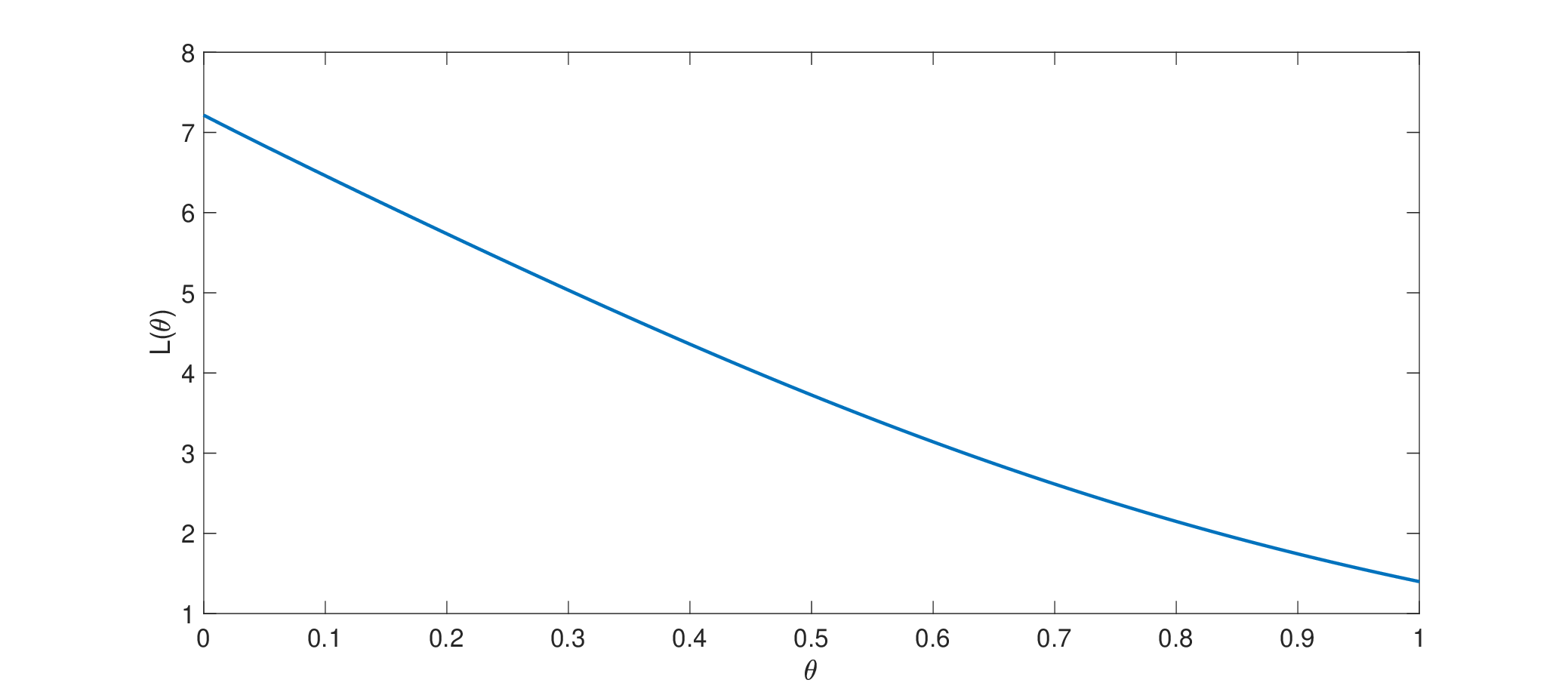}
		\caption{\footnotesize \textbf{\emph{The likelihood function of the parameter $\theta$ given the sample $X_{1},X_{2},X_{3}$.}}}
		\label{fig:ch3ex2}
	\end{figure}
So, the MLE in of the parameter $\theta$ is $0$, as the likelihood function is decreasing in the interval $[0,1]$.
\end{example}

\begin{example}
	Let $n=3$ and the random sample 
	\[X_{1}=(0.25,0.25)^\top, X_{2}=(0.435077,0.114922)^\top, X_{3}=(0.5,0.7)^\top,\]
	 then the score equation is 
	\[ \frac{f(\theta)}{g(\theta)}=0,\]
	where
	\begin{equation*}
		f(\theta)=0.0005 \theta^{6} -0.0149 \theta^{5} +0.141 \theta^{4} -0.5896 \theta^{3}+ 1.1845 \theta^{2}-1.3419 \theta +1.2125,
	\end{equation*} 
	and 
	\begin{equation*}
		g(\theta)=0.0011 \theta^{6} -0.018 \theta^{5} +0.1106 \theta^{4} -0.3181 \theta^{3}+ 0.4975 \theta^{2}-0.75 \theta +1.
	\end{equation*} 
    $V(g_{1},g_{2})=\{4\}$, and $V(f_{1},g_{1})=\{4\}$, then by \hyperref[thm:2]{Theorem \ref{thm:2}}, $V(f,g)\ne \emptyset$.
    
	Both polynomials$f(\theta)$ and $g(\theta)$ are of degree $6$, and 
	\[V(f)=\{15.9985, 4.7451, 4, 0.3290\pm 1.3657i\} \]
	\[V(g)=\{5,4,-0.2857\pm 1.666i \},\]
	and $V(f,g)=\{4\}$ with multiplicity $2$, and $4$ is root of $g(\theta)$ with multiplicity $3$. Therefore, from \hyperref[thm:15]{Theorem \ref{thm:15}}, the ML-degree of the parameter $\theta$ is $4$. The solutions of the score equation are not in the parameter space $[0,1]$. Thus, in that case the likelihood function is either increasing or decreasing. The likelihood function of $\theta$ given the random sample $X_{1},X_{2},X_{3}$ is
	\begin{equation*}
		L(\theta\mid X_{1},X_{2},X_{3})= e^{3.75-(0.4625) \theta} g(\theta),
	\end{equation*}
	given as follow: 
	\begin{figure}[H]
		\centering
		\includegraphics[width=0.9\linewidth]{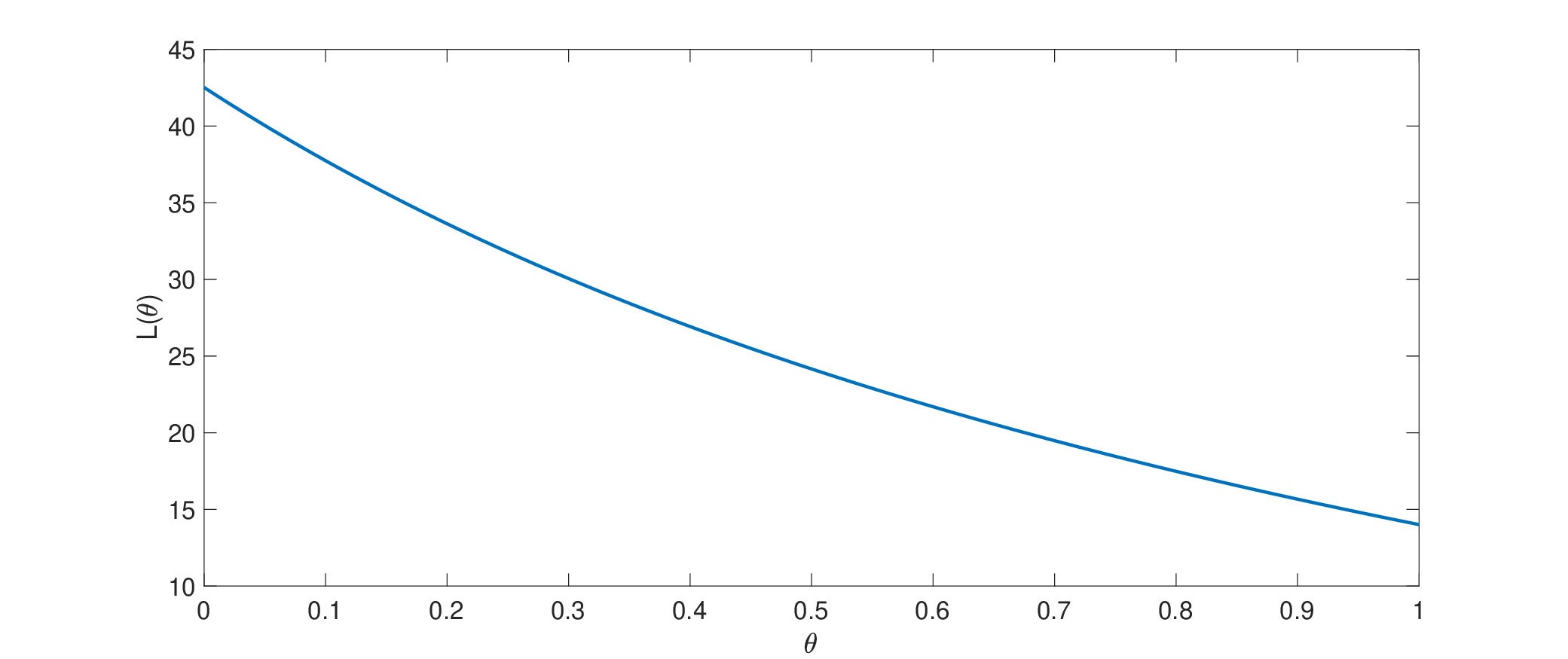}
		\caption{\footnotesize \textbf{\emph{The likelihood function of the parameter $\theta$ given the sample $X_{1},X_{2},X_{3}$.}}}
		\label{fig:ch3ex3}
	\end{figure}
	So, the MLE in of the parameter $\theta$ is $0$, as the likelihood function is decreasing in the interval $[0,1]$.
\end{example}

\section{Conclusion}
	In this paper, the ML-degree of the association parameter in GBED($\theta$) is investigated. The ML-degree of $\theta$ is the number of solutions of the score equation \eqref{eq:2} counted with multiplicity over the complex field. The score equation \eqref{eq:2} is a rational function in $\theta$, so some algebraic techniques are used to find the ML-degree. Since the score equation \eqref{eq:2} can be written as a ratio of two algebraic functions, that is, $\frac{f(\theta)}{g(\theta)}=0$, with $f(\theta)$ and $g(\theta)$ as given in \eqref{eq:5} and \eqref{eq:6}. The solutions of the score equation are in $V(f)$, but this variety may contain some elements where the score equation is not defined because that may be zero of the denominator function $g(\theta)$. 
	
	Thus, for calculation of the ML-degree of $\theta$, the elements of $V(f,g)$ should be removed from $V(f)$, and it will be the number of elements in $V(f) \setminus V(f,g)$, counted with multiplicity. To examine $V(f,g)$, the geometry of the score equation is discussed in section 3, where \hyperref[thm:2]{Theorem \ref{thm:2}} gives the conditions for $V(f,g)$ to be non-empty. Subsequently, more conditions for $V(f,g)$ to be non-empty are provided with a few more results.
	
	Since for computing the ML-degree of $\theta$, the multiplicity of each element of $V(f) \setminus V(f,g)$ is required, but the elements of $V(f) \setminus V(f,g)$ are not known. Hence, the multiplicity of the elements of $V(f,g)$ is counted in $V(f)$, and it is subtracted from $2n$ for computing the ML-degree of $\theta$. In section 4, the multiplicity of the elements of $V(f,g)$ is counted in $V(f)$ for all possible cases. Finally, in section 5, the ML-degree of the association parameter $\theta$ $(m_{d}(\theta))$ in GBED($\theta$) is computed for all possible cases in Theorems \ref{thm:11} to \ref{thm:15}.
	
	If $V(f,g)=\emptyset$, then $m_{d}(\theta)=2n$, as  $f(\theta)$ is a polynomial of degree $2n$, which is maximum. The ML-degree of $\theta$ will be minimum when the multiplicity of elements of $V(f,g)$ is maximum in $f(\theta)$. If $V(g_{i},g_{j})\ne \emptyset$ $\forall i\ne j$, then $m_{d}(\theta)=n+1$ or $n$, depending on if none of the $g_{i}(\theta)$ have double zeros or only one $g_{i}(\theta)$ has double zeros as shown in Theorems \ref{thm:11} and \ref{thm:12}, respectively. Observe that \hyperref[thm:11]{Theorem \ref{thm:11}} has one additional condition on sample size as $n\ge 4$. Further, if $V(g_{i},g_{j})=\emptyset$ $\forall i\ne j$, then $m_{d}(\theta)=2n-n_{1}$, where $n_{1}(\le n)$ $g_{i}(\theta)$ have double zeros, which gives minimum value for $m_{d}(\theta)$ as $n$.
	
	In \hyperref[thm:14]{Theorem \ref{thm:14}}, if some of the $g_{i}(\theta)$ have common zeros and none of the $g_{i}(\theta)$ have double zero, then $m_{d}(\theta)=2n+l-m$, where $g(\theta)= \prod_{i=1	}^{n} g_{i}(\theta)$ has exactly $l$ $(\le n)$ repeated zeros with total multiplicity $m$ $(\le2n)$. Here, $m=2n$ if $l=n$ (trivial case), but $m$ can be $2n$ for $3\le l < n$, giving $m_{d}(\theta)$ as lower as $l$, the number of repeated zeros of $g(\theta)$. Further, in \hyperref[thm:15]{Theorem \ref{thm:15}}, it is shown that $m_{d}(\theta)=2n-m$ if some $g_{i}(\theta)$ has double zero, which are also zeros of other $g_{j}(\theta)$ (at least one). As discussed after \hyperref[thm:15]{Theorem \ref{thm:15}}, the number of $g_{i}(\theta)$ having double zeros (say $l$) does not exceed $\frac{2n}{3}$ and $m\le 2n-l$. If $l=1$, then by \hyperref[thm:10]{Lemma \ref{thm:10}}, $m\le n$, hence $m_{d}(\theta)\ge n$, and for $l\ge 2$, $m\le 2n-l$, thus $m_{d}(\theta)\ge l \ge 2$, the equality depends on $n$ (sample size).
	
	The maximum likelihood estimator of $\theta$ is one of the solutions of the score equation that lie in the interval $[0,1]$, but the score equation has many complex solutions (more than or equal to $2$ depending on $n$). Still, the number of real solutions among them remains unknown and requires further investigation. A closed-form expression for the maximum likelihood estimate of the association parameter can not be obtained, as the score equation is the sum of rational functions over each data point.

\section*{Acknowledgements}The first author would like to thank the University Grants Commission, India, for providing financial support.


\begin{thebibliography}{90}
	
	\bibitem{ABBG}C. Améndola, N. Bliss, I. Burke, C. R. Gibbons, M. Helmer, S. Hoşten, E. D. Nash, J. I. Rodriguez and D. Smolkin, The maximum likelihood degree of toric varieties, \emph{Journal of Symbolic Computation}, \textbf{92}, (2019) 222-242.
	\bibitem{Ba}V. Barnett, The bivariate exponential distribution; a review and some new results. \emph{Statistica Neerlandica}, \textbf{39}(4), (1985) 343-356.
	\bibitem{BBa} N. Balakrishnan and A. P. Basu, \emph{Exponential distribution: theory, methods and applications}, Gordon and Breach, Amsterdam, (1995).
	\bibitem{BBe}K. Balasubramanian and M. I.Beg, Concomitant of order statistics in Gumbel's bivariate exponential distribution, \emph{ Sankhyā: The Indian Journal of Statistics, Series B}, (1998) 399-406.
	\bibitem{CHKS}F. Catanese, S. Hoşten, A. Khetan and B. Sturmfels, The maximum likelihood degree, \emph{American Journal of Mathematics}, \textbf{128}(3), (2006) 671-697.
	\bibitem{CLOS}D. Cox, J. Little, D. O'shea and M. Sweedler, \emph{Ideals, varieties, and algorithms (Vol. 3)}, New York: Springer (1997).
	\bibitem{CMR}J. I. Coons, O. Marigliano and M. Ruddy, Maximum likelihood degree of the two-dimensional linear Gaussian covariance model, \emph{Algebraic Statistics}, \textbf{11}(2), (2020) 107-123.
	\bibitem{CSH}E. Castillo, J. M. Sarabia, and A. S. Hadi, Fitting continuous bivariate distributions to data, \emph{Journal of the Royal Statistical Society: Series D (The Statistician)}, \textbf{46}(3), (1997) 355-369.
	\bibitem{DNR}G. R. Dargahi-Noubary and M. Razzaghi, Earthquake hazard assessment based on bivariate exponential distributions, \emph{Reliability Engineering and System Safety}, \textbf{44}(2), (1994) 153-166.
	\bibitem{DSS}M. Drton, B. Sturmfels, and S. Sullivant, \emph{Lectures on algebraic statistics} (Vol. 39), Springer Science and Business Media, (2008). 
	\bibitem{Gu}E. J. Gumbel, Bivariate exponential distributions, \emph{Journal of the American Statistical Association}, \textbf{55}(292), (1960) 698-707.
	\bibitem{Ha}R. Harris,  Reliability applications of a bivariate exponential distribution, \emph{Operations Research}, \textbf{16}(1), (1968) 18-27.
	\bibitem{KBJ}S. Kotz, N. Balakrishnan and N. L. Johnson,  \emph{Continuous multivariate distributions, Volume 1: Models and applications},  John Wiley and Sons, (Vol. 334) (2019).
	\bibitem{LB}C. D. Lai and N. Balakrishnan, \emph{Continuous bivariate distributions}, Springer-Verlag New York, (2009).
	\bibitem{LNRW}J. Lindberg, N. Nicholson, J. I. Rodriguez and Z. Wang, The maximum likelihood degree of sparse polynomial systems, \emph{SIAM Journal on Applied Algebra and Geometry}, \textbf{7}(1), (2023) 159-171.
	\bibitem{MO}A. W. Marshall and I. Olkin, A multivariate exponential distribution, \emph{Journal of the American Statistical Association}, \textbf{62}(317), (1967) 30-44.
	\bibitem{MMW}M. Michałek, L. Monin and J. A. Wisniewski, Maximum likelihood degree, complete quadrics, and C*-action, \emph{SIAM journal on applied algebra and geometry}, \textbf{5}(1), (2021) 60-85.
	\bibitem{NK}S. Nadarajah and S. Kotz, Reliability for some bivariate exponential distributions, \emph{Mathematical Problems in Engineering}, 2006(1), (2006) 041652.
	\bibitem{Pa}Y. Pawitan, \emph{In all likelihood: statistical modelling and inference using likelihood}, Oxford University Press, (2001). 
	\bibitem{PM}S. Pal and G. S. R. Murthy, An application of Gumbel's bivariate exponential distribution in estimation of warranty cost of motor cycles, \emph{International Journal of Quality and Reliability Management}, \textbf{20}(4), (2003) 488-502.
	\bibitem{S}S. Sullivant, \emph{Algebraic Statistics}, American Mathematical Soc., \textbf{194} (2018).
	\bibitem{SY}Y. C. Sevil and T. O. Yildiz, Gumbel’s bivariate exponential distribution: estimation of the association parameter using ranked set sampling, \emph{Computational Statistics}, (2022) 1-32.
	
	
\end{thebibliography}
\end{document}